\documentclass[12pt]{amsart}
\usepackage{amsmath,amsfonts,amsthm,amscd,amssymb,mathrsfs,amssymb}
\usepackage{stmaryrd}
\usepackage{graphicx}
\newtheorem{theorem}{Theorem}

\newtheorem{proposition}[theorem]{Proposition}

\newtheorem{definition}[theorem]{Definition}

\newtheorem{corollary}[theorem]{Corollary}

\newtheorem{remark}[theorem]{Remark}

\newcommand{\CP}{\mathbb{CP}}
\newcommand{\CC}{\mathbb{C}}
\newcommand{\HH}{\mathbb{H}}
\newcommand{\RR}{\mathbb{R}}
\newcommand{\ZZ}{\mathbb{Z}}

\newcommand{\LB}{{\rm{LB}}}

\numberwithin{equation}{section}
\numberwithin{theorem}{section}
\numberwithin{table}{section}
\begin{document}
\bibliographystyle{amsalpha} 
\title[Index theorem]{An index theorem on anti-self-dual orbifolds}
\author{Jeff A. Viaclovsky}
\address{Department of Mathematics, University of Wisconsin, Madison, 
WI, 53706}
\email{jeffv@math.wisc.edu}
\thanks{Research partially supported by NSF Grants DMS-0804042 and DMS-1105187}
\begin{abstract}
An index theorem for the 
anti-self-dual deformation complex on anti-self-dual orbifolds
with singularities conjugate to ADE-type is proved. 
In 1988, Claude Lebrun gave examples of scalar-flat 
K\"ahler ALE metrics with negative mass, on the total 
space of the bundle $\mathcal{O}(-n)$ over $S^2$. 
A corollary of this index theorem is that the moduli space of 
anti-self-dual ALE metrics near each of these metrics
has dimension at least  $4n-12$, and thus for $n \geq 4$ 
the LeBrun metrics admit a plethora of non-trivial anti-self-dual deformations.  
\end{abstract}
\date{February 2, 2012. Revised May 2012.}
\maketitle

\section{Introduction}

If $(M^{4},g)$ is an oriented four-dimensional Riemannian manifold, the
Hodge star operator associated to $g$ acting on $2$-forms is a mapping
 $*:\Lambda^{2} \mapsto\Lambda^{2}$ satisfying $*^2 = Id$, 
and $\Lambda^{2}$ admits a decomposition of the form 
\begin{align}\label{2formsplusminus}
\Lambda^{2}=\Lambda^{2}_{+} \oplus\Lambda^{2}_{-},
\end{align}
where $\Lambda^{2}_{\pm}$ are the $\pm 1$ eigenspaces of 
$\displaystyle{*|}_{\Lambda^{2}}$. 
The Weyl tensor can be viewed as an operator 
$\mathcal{W}_g: \Lambda^2 \rightarrow \Lambda^2$, and we define
$\mathcal{W}^{\pm}_g = \pi_{\pm} \mathcal{W}_g \pi_{\pm}$, 
where $\pi_{\pm}$ is the projection onto 
$\Lambda^2_{\pm}$.

\begin{definition}{\em
Let $(M^4,g)$ be an oriented four-manifold. 
The metric $g$ is called {\em{anti-self-dual}} if $\mathcal{W}^+_g = 0$. 
}
\end{definition}
 Since Poon's example of a $1$-parameter family of anti-self-dual 
metrics on $\overline{\mathbb{CP}}^2 \# \overline{\mathbb{CP}}^2$ \cite{Poon}, 
a large number of examples of anti-self-dual 
metrics on various four-manifolds have been found.
We do not attempt to 
give a complete history here, and only mention following works
\cite{DonaldsonFriedman, Floer, Honda28, Joyce1995, KovalevSinger, 
LeBrunJDG, LeBrunSinger}.
In this paper, we will be concerned with orbifold 
metrics in dimension four with isolated orbifold points:
\begin{definition}{\em
 A {\em{Riemannian 
orbifold}} $(M^4,g)$ is a topological space which is a 
smooth manifold of dimension $4$ with a smooth Riemannian metric 
away from finitely many singular points.  
At a singular point $p$, $M$ is locally diffeomorphic 
to a cone $\mathcal{C}$ on 
$S^{3} / \Gamma$, where $\Gamma \subset {\rm{SO}}(4)$ 
is a finite subgroup acting freely on $S^{3}$.
Furthermore, at such a singular point, the metric is locally the 
quotient of a smooth $\Gamma$-invariant metric on $B^4$ under 
the orbifold group $\Gamma$.}
\end{definition}
 Given a compact Riemannian orbifold $(\hat{M}, \hat{g})$ with 
non-negative scalar curvature, 
one can use the Green's function for the conformal 
Laplacian $G_p$ to 
associate with any point $p$ a non-compact scalar-flat orbifold by 
\begin{align}
(M \setminus \{p\}, g_p = G_p^{2}\hat{g} ).  
\end{align}
A coordinate system at infinity arises from using inverted 
normal coordinates in the metric $g$ in a neighborhood
of the point $p$, which gives rise to the following 
definition: 
\begin{definition}
\label{ALEdef}
{\em
 A complete Riemannian manifold $(X^4,g)$ 
is called {\em{asymptotically locally 
Euclidean}} or {\em{ALE}} of order $\tau$ if 
there exists a finite subgroup 
$\Gamma \subset {\rm{SO}}(4)$ 
acting freely on $S^3$ and a 
diffeomorphism 
$\psi : X \setminus K \rightarrow ( \mathbf{R}^4 \setminus B(0,R)) / \Gamma$ 
where $K$ is a compact subset of $X$, and such that under this identification, 
\begin{align}
\label{eqgdfale1in}
(\psi_* g)_{ij} &= \delta_{ij} + O( \rho^{-\tau}),\\
\label{eqgdfale2in}
\ \partial^{|k|} (\psi_*g)_{ij} &= O(\rho^{-\tau - k }),
\end{align}
for any partial derivative of order $k$, as
$r \rightarrow \infty$, where $\rho$ is the distance to some fixed basepoint.  
}
\end{definition}

By an {\em{orbifold compactification}} of an ALE space $(X,g)$,
we mean choosing a conformal factor $u : X \rightarrow \RR_+$ 
such that $u = O(\rho^{-2})$ as $\rho \rightarrow \infty$. 
The space $(X, u^2 g)$ then compactifies to a
$C^{1,\alpha}$ orbifold. In the anti-self-dual case, there 
moreover exists a $C^{\infty}$-orbifold conformal compactification 
$(\hat{X}, \hat{g})$ with positive Yamabe invariant \cite[Proposition 12]{CLW}. 
So from the conformal perspective, anti-self-dual ALE spaces are more or less 
the same as anti-self-dual Riemannian orbifolds. 

 There are many interesting examples of 
anti-self-dual ALE spaces.
Eguchi-Hanson discovered 
a Ricci-flat anti-self-dual metric on $\mathcal{O}(-2)$ which is ALE with 
group $\ZZ / 2 \ZZ$ at infinity \cite{EguchiHanson}. 
Gibbons-Hawking then wrote down a metric ansatz depending on the choice of 
$n$ monopole points in $\RR^3$, giving an anti-self-dual ALE hyperk\"ahler metric 
with group $\ZZ / n \ZZ$ at infinity, which 
are called multi-Eguchi-Hanson metrics \cite{GibbonsHawking, Hitchin2}. 
In 1989, Kronheimer then classified all hyperk\"ahler ALE spaces
in dimension $4$, \cite{Kronheimer, Kronheimer2}, 
which we will describe in Section~\ref{S1}.
Using the Joyce construction from \cite{Joyce1995}, 
Calderbank and Singer wrote down many examples of toric 
ALE anti-self-dual metrics, which are moreover 
scalar-flat K\"ahler, and have cyclic groups $\ZZ/ n\ZZ$ at 
infinity contained in ${\rm{U}}(2)$ \cite{CalderbankSinger}.
\subsection{Group actions}
We will next consider the following subgroups of ${\rm{SU}}(2)$:

\begin{itemize}
\item
Type $A_n, n \geq 1$: $\Gamma$ the cyclic group $\mathbb{Z}_{n+1}$,
\begin{align}
\label{su2}
\left(
\begin{matrix}
\exp^{2 \pi i p / (n+1)}   &  0 \\
0    & \exp^{-2 \pi i p / (n+1) } \\
\end{matrix}
\right),  \ \ 0 \leq p \leq n. 
\end{align}
acting on $\RR^4$, which is identified with $\CC^2$
via the map  
\begin{align}
(x_1, y_1, x_2, y_2) \mapsto (x_1 + i y_1, x_2 + i y_2) = (z_1, z_2).
\end{align}
Writing a quaternion $q \in \HH$ as $\alpha + \hat{j} \beta$ for 
$\alpha, \beta \in \CC$, we can also describe the action as 
generated by $e^{2\pi i / n}$, acting on the left. 
\item
Type $D_{n}, n \geq 3$: $\Gamma$ the binary dihedral group $\mathbb{D}^*_{n-2}$
of order $4(n-2)$. This is generated by $e^{\pi i / (n-2)}$ and $\hat{j}$, 
both acting on the left. 
\item
Type $E_6: \Gamma= \mathbb{T}^*$, 
the binary tetrahedral group of order $24$, double cover of~$A(4)$. 
\item
Type $E_7: \Gamma= \mathbb{O}^*$, 
the binary octohedral group of order $48$, double cover of~$S(4)$.
\item
Type $E_8: \Gamma = \mathbb{I}^*$, 
the binary icosahedral group of order $120$, double cover of~$A(5)$. 
\end{itemize}
The ``type'' terminology arises from the relation with 
hyperk\"ahler ALE spaces, which will be discussed in Section \ref{S1}.
Next, we have the notion of conjugate group actions: 
\begin{definition}{\em
A group action $\Gamma_1 \subset {\rm{SO}}(4)$ is {\em{conjugate}}
to another group action  $\Gamma_2 \subset {\rm{SO}}(4)$ if 
there is an intertwining map between the corresponding representations. 
That is, writing the $\Gamma_i$-action as a map 
$F_i : \Gamma_1 \rightarrow {\rm{SO}}(4)$ for $i = 1, 2$, 
then there exists an element $O \in O(4)$ 
such that $F_1 \circ O = O \circ F_2$. 
If $O \in {\rm{SO}}(4)$, then $\Gamma_1$ and $\Gamma_2$ are said to be 
orientation-preserving conjugate, while if $O \notin {\rm{SO}}(4)$, 
then $\Gamma_1$ and $\Gamma_2$ are said to be 
orientation-reversing conjugate
}
\end{definition}

When $\Gamma$ is not a cyclic group, any two subgroups of ${\rm{SO}}(4)$ 
which are isomorphic to $\Gamma$ are in fact conjugate \cite{MCC}. 
However, in the case of the cyclic group, there can be many conjugacy 
classes, and in this paper the only cyclic groups actions we consider 
are those conjugate to the $A_n$-type.
Type $D_3$ is in fact orientation-preserving 
conjugate to type $A_3$.

We note the important fact that if $(X,g)$ is an anti-self-dual 
ALE space with group $\Gamma$ at 
infinity, then the conformal compactification $(\hat{X}, \hat{g})$ with the 
anti-self-dual orientation has group 
$\tilde{\Gamma}$ at the orbifold point where $\tilde{\Gamma}$ is 
orientation-reversing conjugate to $\Gamma$. 
\subsection{Orbifold index theorems}

 Anti-self-dual metrics have a rich obstruction theory. If 
$(M,g)$ is an anti-self-dual four-manifold, the deformation complex is given by 
\begin{align}
\label{thecomplex}
\Gamma(T^*M) \overset{\mathcal{K}_g}{\longrightarrow} 
\Gamma(S^2_0(T^*M))  \overset{\mathcal{D}}{\longrightarrow}
\Gamma(S^2_0(\Lambda^2_+)),
\end{align}
where $S^2_0$ denotes traceless symmetric tensors, 
$\mathcal{K}_g$ is the conformal Killing operator defined 
by 
\begin{align}
( \mathcal{K}_g(\omega))_{ij} = \nabla_i \omega_j + \nabla_j \omega_i - 
\frac{1}{2} (\delta \omega) g, 
\end{align}
with $\delta \omega = \nabla^i \omega_i$, 
and $\mathcal{D} = (\mathcal{W}^+)_g'$ is the linearized self-dual Weyl curvature 
operator.

For a compact smooth closed manifold, there is a formula for the index depending 
only upon topological quantities. 
Let us denote by 
\begin{align}
Ind(M, g) = \dim( H^0(M,g)) -  \dim( H^1(M,g)) + \dim( H^2(M,g)),
\end{align}
where $H^i(M,g)$ is the $i$th cohomology of the complex \eqref{thecomplex}, 
for $i = 0,1,2$. For a compact anti-self-dual metric, we have 
\begin{align}
Ind(M, g) = \frac{1}{2} ( 15 \chi(M) + 29 \tau(M)), 
\end{align}
where $\chi(M)$ is the Euler characteristic and 
$\tau(M)$ is the signature of $M$. This formula 
is proved in \cite{KotschickKing}, but 
was also known to some experts before that paper, see for example
\cite[equation (1.2)]{Floer}, and \cite[page 369]{EGH} where it is attributed to I.M. Singer in 1978.

Our first result is an index theorem for an anti-self-dual orbifold with a singularity
{\em{orientation-reversing}} conjugate to ADE-type:
\begin{theorem}
\label{mainit2}
Let $(\hat{M}, \hat{g})$ be a compact anti-self-dual orbifold with a single orbifold 
point~$p$ with orbifold group $\Gamma$. If $\Gamma$ 
is  orientation-reversing conjugate to type~$A_1$, then 
\begin{align}
Ind(\hat{M}, \hat{g}) = \frac{1}{2} ( 15 \chi(\hat{M}) + 29 \tau (\hat{M})) 
-4.
\end{align}
If $\Gamma$ is orientation-reversing conjugate to type 
$A_{n}$ with $n \geq 2$, then 
\begin{align}
Ind(\hat{M}, \hat{g}) = \frac{1}{2} ( 15 \chi(\hat{M}) + 29 \tau (\hat{M})) 
+ 4  n - 10.
\end{align}
If $\Gamma$ is orientation-reversing conjugate to type $D_3$, then 
\begin{align}
Ind(\hat{M}, \hat{g}) = \frac{1}{2} ( 15 \chi(\hat{M}) + 29 \tau (\hat{M})) +2.
\end{align}
If $\Gamma$ is orientation-reversing conjugate to type
$D_{n}$ with $n \geq 4$, or type $E_{n}$ with $n = 6, 7, 8$, then 
\begin{align}
Ind( \hat{M}, \hat{g}) =  \frac{1}{2} ( 15 \chi(\hat{M}) + 29 \tau (\hat{M})) 
+ 4n - 11.
\end{align}
\end{theorem}

The next result is an index theorem for an anti-self-dual
orbifold with a singularity {\em{orientation-preserving}} conjugate to ADE-type: 
\begin{theorem}
\label{mainit}
Let $(\hat{M}, \hat{g})$ be a compact anti-self-dual orbifold with a single orbifold 
point~$p$ with orbifold group $\Gamma$ orientation-preserving conjugate to 
type $A_n$ with $n \geq 1$, or $D_n$ with $n \geq 3$, or $E_n$ with $n = 6, 7, 8$.
Then 
\begin{align}
Ind(\hat{M}, \hat{g})  = \frac{1}{2} ( 15 \chi(\hat{M}) + 29 \tau(\hat{M})) - 4 n. 
\end{align}

\end{theorem}
The proofs of Theorem~\ref{mainit2} and Theorem~\ref{mainit} 
use Kawasaki's orbifold index theorem \cite{Kawasaki}. 
However, we do not compute the correction terms directly, 
but instead use an analytic method to determine the 
correction terms using certain examples.\footnote{The correction term 
for any cyclic quotient singularity 
has recently been computed in \cite{LockViaclovsky}.}

\begin{remark}
\label{srmk}
{\em
For simplicity, the above theorems are stated in the case of a single 
orbifold point. However, if there are several orbifold points each of 
the above types, then a similar formula holds, with 
the correction term simply the sum of the corresponding 
correction terms for each type of orbifold point. }
\end{remark}

\subsection{LeBrun negative mass metrics}
\label{negativemass}
In \cite{LeBrunnegative}, LeBrun presented the first known examples of scalar-flat ALE 
spaces of negative mass, which gave counterexamples to extending the
positive mass theorem to ALE spaces. We briefly describe these as follows. Define
\begin{align}
g_{\LB} = \frac{ dr^2}{ 1 + Ar^{-2} + B r^{-4}} +r^2 \Big[ \sigma_1^2 + \sigma_2^2
+ (  1 + Ar^{-2} + B r^{-4}) \sigma_3^2 \Big],
\end{align}
where $r$ is a radial coordinate, and $\{ \sigma_1, \sigma_2, \sigma_3 \}$ is a 
left-invariant coframe on $S^3 = {\rm{SU}}(2)$, and $A = n -2$, $B = 1 - n$. 
Redefine the radial coordinate to be $\hat{r}^2 = r^2 - 1$, 
and attach a $\CP^1$ at $\hat{r} = 0$. After taking a quotient by $\ZZ_{n}$, 
with action given by the diagonal action
\begin{align}
\label{u2c}
(z_1, z_2) \mapsto \exp^{2 \pi i p / n}(z_1, z_2),  \ \ 0 \leq p \leq n -1,
\end{align}
the metric then extends smoothly over the added $\CP^1$, is ALE at infinity,
and is diffeomorphic to $\mathcal{O}(-n)$.  
The mass is computed to be $-4 \pi^2 (n -2)$, which is 
negative when $n > 2$. 
These metrics are scalar-flat K\"ahler, and satisfy 
$b^2_- = 1, \tau = -1, \chi = 2$.

Since the $\ZZ/n \ZZ$-action in \eqref{u2c} is orientation-reversing conjugate to 
type $A_{n-1}$ under the intertwining map $(z_1, z_2) 
\mapsto (z_1, \overline{z}_2)$, a corollary of Theorem \ref{mainit} 
is the following:
\begin{corollary}
\label{lbcor}
Let $( \widehat{\mathcal{O}(-n)}, \hat{g}_{\LB})$ be a conformally 
compactified LeBrun metric. Then 
\begin{align}
\label{lbind}
Ind ( \widehat{\mathcal{O}(-n)}, \hat{g}_{\LB}) = 12 - 4n. 
\end{align}
\end{corollary}

We briefly recall some details of moduli space theory 
\cite{Itoh2, KotschickKing} (these references deal with the 
case of smooth manifolds, but the proofs are easily generalized to the
setting of orbifolds).  
Given an anti-self-dual metric $g$ on a compact orbifold, there 
is a map $\Psi: H^1 \rightarrow H^2$, called the {\em{Kuranishi map}}
which is equivariant with respect to the action of $H^0$, 
and the moduli space of anti-self-dual conformal structures 
near $g$ is locally isomorphic to 
$\Psi^{-1}(0) / H^0$. Therefore, if $H^2 = 0$, the moduli 
space is locally isomorphic to $H^1 / H^0$. 

Our final result is about the moduli space of anti-self-dual metrics nearby the
conformally compactified LeBrun 
negative mass metrics. The case $n = 1$
is the Burns metric, which is conformal to Fubini-Study metric 
on $\CP^2$, and is rigid. The case $n =2$ is the Eguchi-Hanson metric which 
is also rigid. But for $n \geq 4$ these metrics are not rigid 
as anti-self-dual metrics:
\begin{theorem} 
\label{defthm}
For $n \geq 4$, the dimension of the 
moduli space of anti-self-dual orbifold metrics near a LeBrun 
metric $( \widehat{\mathcal{O}(-n)}, \hat{g}_{\LB})$ is at least $4n - 12$. 
\end{theorem}
This is proved in Section \ref{exps} using the above index theorems.
It is easy to see that any sufficiently close deformed metric has positive 
orbifold Yamabe invariant, and thus there is an associated 
anti-self-dual ALE space \cite{AB2, ViaclovskyFourier}. Thus the above theorem could equivalently be stated in terms 
of the moduli space of anti-self-dual ALE metrics near the ALE metric $( \mathcal{O}(-n), g_{\LB})$.

To exactly determine the dimension of the moduli space near the 
LeBrun metrics, it would be necessary to explicitly compute the 
action of $H^0$ on $H^1$. This does not follow from the 
above index theorems, which is why we can only give a lower bound for the 
dimension of the moduli space.  

\subsection{Questions}
We end the introduction with some interesting 
questions:\footnote{Questions 2--4 have recently been answered by 
Nobuhiro Honda using arguments from twistor theory. We refer the reader to 
\cite{HondaOn} for the complete statement of his result.}
\begin{enumerate}

\vspace{3mm}
\item The hyperk\"ahler ALE spaces are anti-self-dual spaces with group 
actions of type $A_n$, $D_n$, $E_6$, $E_7$, and $E_8$
contained in ${\rm{SU}}(2)$. 
The LeBrun negative mass metrics are examples of anti-self-dual ALE spaces with 
group orientation-reversing conjugate to type $A_n$. 
Are there in fact non-trivial examples of anti-self-dual ALE spaces 
with group at infinity {\em{orientation-reversing}} conjugate 
to the $D_n$ type for $n \geq 3$, and to the types $E_6, E_7$ and $E_8$?

\vspace{3mm}
\item  The paper \cite{PedersenPoonKSF} 
discusses some K\"ahler scalar-flat deformations of the LeBrun 
negative mass metrics. But there is no indication given there 
of the dimension of such deformations; it is not clear what 
free parameters there are in this family. Of the $4n-12$ dimensional 
family found above, how many of these are K\"ahler scalar-flat deformations?

\vspace{3mm}
\item  What are the possible conformal automorphism groups of the metrics in the 
$4n-12$ dimensional family found above?

\vspace{3mm} 
\item
For the LeBrun negative mass metrics on $\mathcal{O}(-n)$, what 
is the local dimension of the moduli space of anti-self-dual ALE metrics 
for $n \geq 3$? Is it equal to $4n -12$? Are there nontrivial deformations 
for $n =3$?

\end{enumerate}


\subsection{Acknowledgements} 
The author would like to thank Nobuhiro Honda and Claude LeBrun for numerous enlightening
discussions regarding deformation theory of anti-self-dual metrics. 
Matthew Gursky provided crucial help with Proposition~\ref{alelem}.
The author would also like to thank 
Michael Lock and John Lott for many useful discussions about index theory.

\section{Hyperk\"ahler ALE spaces}
\label{S1}

The hyperk\"ahler ALE spaces were classified in dimension 
$4$ by Kronheimer \cite{Kronheimer, Kronheimer2}. 
These are anti-self-dual Ricci flat-ALE metrics of order $4$, 
with groups at infinity of ADE-type described in the 
introduction. 
We write the three independent complex structures as $I, J, K$.
Using the metric, these are identified with K\"ahler forms 
$\omega_I, \omega_J, \omega_K$, which are parallel self-dual $2$-forms. 
The cohomology of these spaces are generated by $2$-spheres with
self-intersection~$-2$, with
intersection matrix given by the negative of the 
corresponding Cartan matrix. 
We summarize the above in Table~\ref{hktp}.
\begin{table}[t]
\caption{Invariants of hyperk\"ahler ALE spaces.}
\label{hktp}
\begin{tabular}{l l l l l}
\hline
Type & $\Gamma$ & $|\Gamma|$& $ b_2^-$ & $\chi$ \\
\hline
$A_n, n \geq 1$ & $\mathbb{Z}_{n+1}$ & $n+1$ & $n$ & $n+1$  \\
$D_m, m \geq 3$ & $\mathbb{D}^*_{m-2}$ & $4(m-2)$ & $m$ & $m+1$ \\
$E_6$ & $ \mathbb{T}^*$ & $24$ &  $6$ & $7$ \\
$E_7$ &  $\mathbb{O}^*$  & $48$ & $7$ & $8$ \\
$E_8$ & $\mathbb{I}^*$  & $120$ & $8$ & $9$ \\
\hline
\end{tabular}
\end{table}

We next have a proposition regarding infinitesimal 
deformations of hyperk\"ahler ALE spaces. Notice 
that these spaces are anti-self-dual with the 
complex orientation.

\begin{proposition}
\label{hkprop}
Let $(X,g)$ be a hyperk\"ahler ALE space 
of type $A_n$ for $n \geq 1$, type $D_n$ for $n \geq 3$, 
or type $E_n$ for $n = 6, 7, 8$.  
For $-4 \leq \epsilon < 0$, let $H^1_{\epsilon}(X,g)$ denote the space 
of traceless symmetric $2$-tensors $h \in S^2_0((T^*X))$ satisfying
\begin{align}
(\mathcal{W}^+)_g'(h) = 0,\
\delta_g(h) = 0, 
\end{align}
with $h = O( \rho^{\epsilon})$ as $\rho \rightarrow \infty$,
where $\rho$ is the distance to some fixed basepoint, and 
$(\delta_g h)_i = \nabla^j h_{ij}$ is the divergence. 
Then $H^1_{\epsilon}(X,g) = H^1_{-4}(X,g)$, and
using the isomorphism $S^2_0(T^*M) = \Lambda^2_+ \otimes \Lambda^2_-$,
$H^1_{-4}(X,g)$ has a basis
\begin{align}
\{\omega_I \otimes \omega^-_j, \omega_J \otimes \omega^-_j, 
\omega_K \otimes \omega^-_j\},
\end{align}
where $\{\omega^-_j, j = 1, \dots, n = \dim(H^2(M))\}$ is a basis of
the space of $L^2$ harmonic $2$-forms. 
Consequently, 
\begin{align}
\dim(H^1_{\epsilon}(X,g)) = 3n. 
\end{align}
\end{proposition}
\begin{proof}
We begin by observing that
\begin{align}
B_g'(h) = \Delta_L \Delta_L h = \mathcal{D}_g^* \mathcal{D}_g h, 
\end{align}
where $B_g$ is the Bach tensor. The first identity is 
proved in \cite[Section 3]{GV11} and the second identity is 
proved in \cite[Section 4]{Itoh}. 
Consequently, $h \in Ker (B_g')$ and $\delta_g h = 0$.  
In the traceless divergence-free gauge, 
$B_g'$ is asymptotic to $\Delta^2$ as $\rho \rightarrow \infty$,
so \cite[Proposition 2.2]{AcheViaclovsky} 
implies that there is no $O(\rho^{-1})$ term in the asymptotic expansion of $h$ 
and therefore 
$h = O(\rho^{-2})$ as $\rho \rightarrow \infty$.
Integrating by parts, 
we have that $\Delta_L h = 0$, and 
consequently, $h$ is an infinitesimal Einstein 
deformation. It follows from \cite[Section 5]{ct}, that 
$h = O( \rho^{-4})$ as $\rho \rightarrow \infty$. 
We also see that if $\Delta_L h = 0$ and $\delta_g h= 0$
then $\mathcal{D}(h) = 0$. This shows that decaying infinitesimal Einstein 
deformations are equivalent to decaying infinitesimal anti-self-dual 
deformations on these spaces. 

The identification of the kernel is then 
given by the argument in \cite[Proposition~1.1]{Biquard2011}. 
Briefly, the operator $\Delta_L$ acting on traceless
divergence free tensors can be identified with the 
operator $d_- d_-^*$ where 
\begin{align}
d_+ : \Omega^1 \otimes \Omega^2_+
\rightarrow  \Omega^2_- \otimes \Omega^2_+ \cong \Gamma(S^2_0(T^*X))
\end{align}
is the exterior derivative. Since $\Omega^2_+$ has a basis of 
parallel sections $\{\omega_I, \omega_J, \omega_K\}$, the proposition 
follows since the $L^2$-cohomology $H^2_{(2)}(X)$ is isomorphic to the 
usual cohomology $H^2(X)$ \cite{CarronDuke}. 

\end{proof}
We next write down the index on the conformal compactifications of 
the hyperk\"ahler ALE spaces: 
\begin{theorem}
\label{hypkthm}
Let $(\hat{X}, \hat{g})$ be the conformal compactification
of a hyperk\"ahler ALE space $(X,g)$ with group $\Gamma$ at infinity.
If $\Gamma$ is type $A_1$, then 
\begin{align}
\label{ehi}
Ind( \hat{X}, \hat{g}) = 4.
\end{align}
If $\Gamma$ of type $A_n$ for $n \geq 2$, then
\begin{align}
Ind( \hat{X}, \hat{g}) = -3n +5.
\end{align}
If $\Gamma$ is of type $D_3$, then 
\begin{align}
Ind( \hat{X}, \hat{g}) = -4.
\end{align}
If $\Gamma$ is of type
$D_n$ with $n \geq 4$, or $E_n$ with $n = 6, 7, 8$, then 
\begin{align}
Ind( \hat{X}, \hat{g}) = -3n +4.
\end{align}
\end{theorem}
\noindent
This will be proved in the following section.
\section{Index comparison}
We next have a proposition relating the index on an ALE 
space $(X,g)$ and the index on the compactification $(\hat{X}, \hat{g})$. 
We let $\{z\}$ denote coordinates at infinity for $(X,g)$, let $\rho = |z|$, 
let $\{x\}$ denote coordinate at the orbifold point $p$ of $(\hat{X}, \hat{g})$,
and let $r = |x|$. These satisfy $z = x/ |x|^2$ and $\rho = r^{-1}$. 
We write the metric as $g = G_p^{2} \hat{g}$ where 
$G_p = O(r^{-2})$ as $x \rightarrow 0$. 
Similarly to $H^1_{\epsilon}(X,g)$ defined in Proposition \ref{hkprop}, we define 
 $H^2_{\epsilon}(X,g)$ to be the space of solutions of $\mathcal{D}^*_g Z = 0$ 
satisfying  $Z = O( \rho^{\epsilon})$ as $\rho \rightarrow \infty$. 
We also let  $\dim(H^0(\RR^4/ \Gamma))$ denote the dimension of the 
space of conformal Killing fields on $\RR^4/ \Gamma$ with respect to the
Euclidean metric.
\begin{proposition}
\label{alelem}
Let $(X,g)$ be an anti-self-dual ALE metric with group $\Gamma$ 
at infinity, and let $(\hat{X}, \hat{g})$ 
be the orbifold conformal compactification. 
Then for $-2 < \delta < 0$, we have
\begin{align}
\label{if1}
- \dim(H^1_{\delta}(X, g)) + \dim(H^2_{-2 - \delta}(X, g))
+ \dim(H^0(\RR^4/ \Gamma)) = Ind(\hat{X}, \hat{g}).
\end{align}
\end{proposition}
\begin{proof}
Letting $\mathcal{D}_g$ denote the linearized self-dual Weyl curvature
(viewed as a (1,3)-tensor), we have the conformal transformation formulas
\begin{align}
\label{dchange}
\mathcal{D}_{g} ( h ) = \mathcal{D}_{\hat{g}} (\hat{h}), 
\end{align}
where $\hat{h} = G_p^{-2} h$, 
\begin{align}
\label{dconf}
\mathcal{D}_{g}^* ( Z) = G_p^{-2} \mathcal{D}_{\hat{g}}^* ( \hat{Z}),
\end{align}
where $\hat{Z} = Z$, and 
\begin{align}
\label{kconf}
\mathcal{K}_{g} (\omega) &= G_p^{2} \mathcal{K}_{\hat{g}} ( \hat{\omega}),
\end{align}
where $\hat{\omega} = G_p^{-2} \omega$.

 We note that elementary Fredholm theory shows that 
if $(\hat{X}, \hat{g})$ is a compact anti-self-dual orbifold then the 
cohomology groups of the complex \eqref{thecomplex} are isomorphic 
to the following:
\begin{align}
H^1(X,g) \cong \{ h \in S^2_0(T^*\hat{X}) \ | \ \mathcal{D}_{\hat{g}}(h) = 0,
\ \delta_{\hat{g}}(h) = 0 \}, 
\end{align}
and
\begin{align}
H^2(X,g) \cong \{ Z \in S^2_0( \Lambda^2_-) \ | \ \mathcal{D}_{\hat{g}}^* Z = 0 \}.
\end{align}

We first claim that 
\begin{align}
\label{h2en}
\dim(H^2_{-2 - \delta}(X, g)) = \dim( H^2(\hat{X}, \hat{g}).
\end{align}
To see this, from \cite[Theorem 1.11]{AcheViaclovsky2}, 
if $Z \in H^2_{\epsilon}(X,g)$ for 
$\epsilon < 0$, then $Z \in H^2_{-4}(X,g)$. Using the formula 
\begin{align}
|Z|_g = G_p^{-2} |\hat{Z}|_{\hat{g}}, 
\end{align}
the tensor $\hat{Z}$ is 
then a bounded solution of $\mathcal{D}^*_{\hat{g}} \hat{Z} = 0$ on $\hat{X} \setminus 
\{p\}$. Since $\mathcal{D}_{\hat{g}}\mathcal{D}^*_{\hat{g}}$ is 
an elliptic fourth order operator with leading term $\Delta^2$, $\hat{Z}$ 
extends to a smooth solution on all of $\hat{X}$. 
Conversely, any bounded solution of  $\mathcal{D}^*_{\hat{g}} \hat{Z} = 0$
yields an element $Z \in H^2_{-4}(X,g)$ satisfying 
$\mathcal{D}^*_{g} Z= 0$, and \eqref{h2en} is proved. 

Using the identity \eqref{h2en},  \eqref{if1} is then equivalent to 
\begin{align}
\label{if2}
-\dim(H^1_{\delta}(X, g)) + \dim(H^0(\RR^4/\Gamma)) 
= \dim( H^0 ( \hat{X}, \hat{g})) -  \dim( H^1 ( \hat{X}, \hat{g})),
\end{align}
which we rewrite as 
\begin{align}
\label{if3}
\dim(H^1_{\delta}(X, g))
=  \dim( H^1 ( \hat{X}, \hat{g})) + \{ \dim(H^0(\RR^4/\Gamma)) 
- \dim( H^0 ( \hat{X}, \hat{g}))\}. 
\end{align}
We next claim that there is a surjection
\begin{align}
\label{mapscp}
F: H^1_{\delta}(X, g) \rightarrow  H^1 ( \hat{X}, \hat{g}).
\end{align}
To define this mapping, if $h \in H^1_{\delta}(X, g)$ then $h = O(\rho^{-2})$ 
as $\rho \rightarrow \infty$
by \cite[Theorem 1.11]{AcheViaclovsky2}. From \eqref{dchange},
and the formula 
\begin{align}
\label{hfrm}
|\hat{h}|_{\hat{g}} = |h|_g,
\end{align}
it follows that 
$\hat{h}$ is a solution of $\mathcal{D}_{\hat{g}} ( \hat{h}) = 0$ on $\hat{X} \setminus 
\{p\}$ satisfying $\hat{h} = O(r^2)$ as $r \rightarrow 0$. 
This implies that $\hat{h} \in C^{1,\alpha}(S^2_0(T^*\hat{X}))$, 
so $\delta_{\hat{g}}(\hat{h}) \in  C^{0,\alpha}(T^*\hat{X})$.
Next, consider the operator $\square_{\hat{g}} = \delta_{\hat{g}} \mathcal{K}_{\hat{g}}$
mapping from 
\begin{align}
\square_{\hat{g}} : C^{2,\alpha}(T^* \hat{X}) 
\rightarrow  C^{0,\alpha}(T^* \hat{X}).
\end{align}
This operator is elliptic and self adjoint, with kernel exactly the space of 
conformal Killing fields. 
Since $\delta_{\hat{g}}(\hat{h})$ is orthogonal to this kernel, by Fredholm theory
there exists 
a solution $\hat{\omega} \in C^{2, \alpha}(T^* \hat{X})$  
of the equation $\square_{\hat{g}}\hat{\omega} =\delta_{\hat{g}} \hat{h}$. 
We therefore have the decomposition 
\begin{align}
\label{kdc}
\hat{h} = \mathcal{K}_{\hat{g}} \hat{\omega} + \hat{h}_0,
\end{align}
with $\hat{\omega} \in C^{2, \alpha}(T^* \hat{X})$ and
$\hat{h}_0 \in C^{1,\alpha}(S^2_0(T^* \hat{X}))$ 
satisfying $\delta \hat{h}_0 = 0$. Since 
$\mathcal{D}_{\hat{g}} \mathcal{K}_{\hat{g}} \omega = 0$, we have
$\mathcal{D}_{\hat{g}} \hat{h}_0  = 0$. As mentioned 
in the proof of Theorem \ref{hkprop}, 
$\mathcal{D}^*_{\hat{g}}  \mathcal{D}_{\hat{g}} = B'_{\hat{g}}$, 
has leading term $\Delta^2$ in the traceless divergence-free 
gauge, so the singularity is removable and therefore
$ \hat{h}_0 \in H^1( \hat{X}, \hat{g})$. 
The mapping $F: h \mapsto \hat{h}_0$ is the required mapping in~\eqref{mapscp}. 

We claim that the map $F$ is surjective. To see this, let $\hat{h}_0$ 
satisfy $\delta_{\hat{g}} \hat{h}_0 = 0$ and 
$\mathcal{D}_{\hat{g}} ( \hat{h}_0)  = 0$. Then $h_0 = G_p^2 \hat{h}_0$ 
satisfies $\mathcal{D}_g(h_0) =0$. From $\eqref{hfrm}$ we have 
that $h_0 = O(1)$ as $\rho \rightarrow \infty$, so 
$\delta_g h_0 = O({\rho}^{-1})$ as $\rho \rightarrow \infty$.
Consider 
\begin{align}
\label{sqx}
\square_g : C^{k, \alpha}_{1 + \epsilon} (T^* X) \rightarrow  
C^{k-2, \alpha}_{-1 + \epsilon} (T^* X),
\end{align}
for $\epsilon > 0$ small, and 
where the spaces are weighted H\"older spaces (see \cite{Bartnik}). 
The adjoint mapping has domain weight $-4 - (-1 + \epsilon) = -3 - \epsilon$. 
An integration by parts shows that the 
kernel of the adjoint therefore consists of decaying conformal 
Killing fields, which are necessarily trivial. Consequently, 
the mapping in \eqref{sqx} is surjective. 
So there exists a solution $\omega \in  C^{k, \alpha}_{1 + \epsilon} (T^* X)$
to the equation $\square_g ( \omega) = \delta_g h_0$.  
Defining $\tilde{h}_0 = h_0 - \mathcal{K}_g \omega$, we have 
$\mathcal{D}_g ( \tilde{h}_0) = 0$ and $\delta_g ( \tilde{h}_0) =0$, 
and therefore $\tilde{h}_0 \in H^1_{-2 - \delta}(X,g)$. 
Finally, we have that
\begin{align} 
\hat{\tilde{h}}_0 = G_p^{-2} \tilde{h}_0 = G_p^{-2} (  h_0 - \mathcal{K}_g \omega)
= \hat{h}_0 - \mathcal{K}_{\hat{g}} \hat{\omega}, 
\end{align}
and therefore $F(\tilde{h}_0) = \hat{h}_0$.

We next identify the kernel of the map $F$. If $ \hat{h}_0 = 0$, 
then $h =  \mathcal{K}_{\hat{g}} \hat{\omega}$ for some 
$\hat{\omega} \in C^{2, \alpha}(T^* \hat{X})$. 
The transformation formula \eqref{kconf} implies that 
$h = \mathcal{K}_g ( \omega)$, where $\omega = G_p^2 \hat{\omega}$. 
Since $\hat{\omega}$ satisfies $\hat{\omega} = O(1)$ as $r \rightarrow 0$, 
from the formula 
\begin{align}
\label{ogrow}
|\omega|_g = G_p | \hat{\omega}|_{\hat{g}}, 
\end{align}
it follows that $\omega = O( \rho^2)$ as $\rho \rightarrow \infty$. 
Since $\delta_g h = \square_g \omega = 0$, and $\square_g$ is an 
elliptic operator, $\omega$ admits an asymptotic expansion at 
infinity with leading term a solution of $\square \omega = 0$ in $\RR^4 / \Gamma$. 
To count these solutions, we use the relative index theorem 
of \cite{LockhartMcowen}, which says that for non-exceptional 
weights $\delta_1 < \delta_2$, 
\begin{align}
Ind( \square_g, \delta_2 )  - Ind( \square_g, \delta_1) = N( \delta_1, \delta_2), 
\end{align}
where $N(\delta_1, \delta_2)$ counts the dimension of the space of 
homogeneous solutions in $\RR^4/ \Gamma$ 
with growth rate between $\delta_1$ and $\delta_2$. 
It is easy to see that this implies the following. First, 
\begin{align}
\dim (Ker ( \square_g, \epsilon)) =
\begin{cases}
 4 & \mbox{ if } \Gamma = \{e\} \\
0 & \mbox{ if } \Gamma \mbox{ is nontrivial},\\
\end{cases}
\end{align}
and $\dim (Ker ( \square_g, 1 + \epsilon)) - \dim (Ker ( \square_g, 1- \epsilon))$
is the dimension of the space of $1$-forms with linear coefficients on $\RR^4$ which 
descend to $\RR^4/ \Gamma$. Finally,  
$\dim (Ker ( \square_g, 2 + \epsilon)) - \dim (Ker ( \square_g, 2 - \epsilon))$
is the dimension of the space of $1$-forms $\omega_2$ on $\RR^4$ with quadratic 
coefficients which descend to $\RR^4/ \Gamma$ and which solve $\square \omega_2 = 0$. 

Consequently, we have the following statements:
Given any $1$-form on $\RR^4/ \Gamma$ with constant coefficients $\omega_0$, 
there is a unique solution of $\square_g \omega = 0$ on $(X,g)$ 
with $\omega = \omega_0 + O(\rho^{-1})$ 
as $\rho \rightarrow \infty$. 
Given any $1$-form on $\RR^4/ \Gamma$ with linear coefficients $\omega_1$, 
there is a unique solution of $\square_g \omega = 0$ on $(X,g)$ 
with $\omega = \omega_1 + O(1)$ 
as $\rho \rightarrow \infty$. Finally, given any $1$-form on $\RR^4/ \Gamma$ 
with quadratic coefficients $\omega_2$ satisfying $\square \omega_2 = 0$, 
there is a unique solution of $\square_g \omega = 0$ on $(X,g)$ 
with $\omega = \omega_2 + O(\rho)$ as $\rho \rightarrow \infty$.

However, since $h$ is decaying and $h = \mathcal{K}_g\omega$, the leading 
terms in the asymptotic expansion of $\omega$ must be a conformal 
Killing field in $\RR^4/\Gamma$. 
If the group $\Gamma = \{e\}$, then there is a $15$-dimensional space 
of such solutions.  We are only interested in counting such solutions which do 
not extend to global conformal Killing fields on $(X,g)$.  
Note that from \eqref{kconf}, conformal 
Killing fields on $(X,g)$ correspond exactly to the 
conformal Killing fields on $(\hat{X},\hat{g})$, so the kernel of the map $F$
is of dimension $15 - \dim(H^0(\hat{X}, \hat{g}))
= \dim(H^0(\RR^4)) -  \dim(H^0(\hat{X}, \hat{g}))$. Since $F$ is 
surjective, the theorem follows in this case. 
If the group $\Gamma$ is non-trivial, then the leading term in the asymptotic 
expansion of $\omega$ 
is of the form $c_1 \rho d \rho + \omega_0$, where $\omega_0$ is a 
rotational Killing field on $\RR^4/ \Gamma$, for some constant $c_1$.  
The dimension of the space of such leading terms is 
given by $\dim(H^0(\RR^4/\Gamma))$. 
Again, we are only interested in counting such solutions which do 
not extend to conformal Killing fields on $(X,g)$, so the kernel of $F$ is of 
dimension $\dim(H^0(\RR^4/\Gamma)) -  \dim(H^0(\hat{X}, \hat{g}))$, 
and the proof is complete. 
\end{proof}
\begin{remark}{\em
Formula \eqref{if3} implies that, apart from the case
of $(S^4, g_S)$, the space $H^1_{\delta}(X,g)$ is always strictly 
larger than  $H^1(\hat{X}, \hat{g})$ for $-2 < \delta < 0$. 
The additional kernel elements are of the form $h = \mathcal{K}_g(\omega)$ 
where $\omega$ is a solution of $\square_g \omega = 0$ which is not a 
conformal Killing field. Using the 1-parameter group of diffeomorphisms
generated by $\omega$ it is possible to identify this extra 
kernel with a subspace of ``gluing parameters'' in gluing theory, but we do 
not do this here. 
}
\end{remark}
\begin{proof}[Proof of Theorem \ref{hypkthm}]
Since these spaces are scalar-flat K\"ahler, 
from \cite[Theorem 1.11]{AcheViaclovsky2} and \cite[Theorem 4.2]{LeBrunMaskit}, 
we have that 
$\dim(H^2_{-2 - \delta}(X, g))= 0$. 
Consequently, Proposition \ref{alelem} takes the form 
\begin{align}
\label{ifh}
 Ind(\hat{X}, \hat{g}) = 
- \dim(H^1_{\delta}(X, g)) + \dim(H^0(\RR^4/\Gamma)).
\end{align}
For type $A_n$, if $n =1$, then $\dim(H^0(\RR^4/\Gamma)) = 7$
(this is the dimension of the isometry group plus one for 
the radial scaling), 
so Propositions \ref{hkprop} and \ref{alelem} yield that 
\begin{align}
 Ind(\hat{X}, \hat{g}) = - 3 + 7 = 4. 
\end{align}
If $n \geq 2$, then from \cite[Section 1.3]{MCC} we have
$\dim(H^0(\RR^4/\Gamma)) = 4 +1 = 5$, so  
Propositions \ref{hkprop} and \ref{alelem} yield that 
\begin{align}
 Ind(\hat{X}, \hat{g}) = - 3n + 5.
\end{align}
For type $D_3$, \cite[Section 1.3]{MCC} we have $\dim(H^0(\RR^4/\Gamma)) = 4 +1 = 5$,
so Propositions \ref{hkprop} and \ref{alelem} yield that 
\begin{align}
 Ind(\hat{X}, \hat{g}) = -4.
\end{align}
For type $D_n$, $m \geq 4$, and type $E_n$, $n = 6,7,8$, 
from  \cite[Section 1.3]{MCC} we have
$\dim(H^0(\RR^4/\Gamma)) = 3 +1 = 1$, so  
Propositions \ref{hkprop} and \ref{alelem} yield that 
\begin{align}
 Ind(\hat{X}, \hat{g}) = - 3n + 4.
\end{align}
\end{proof}


\section{Completion of proofs}
\label{exps}

In this section, we complete the proofs of the 
results stated in the Introduction.
\begin{proof}[Proof of Theorem \ref{mainit2}]
From Kawasaki's orbifold index theorem \cite{Kawasaki}, it 
follows that there is an index formula of the form 
\begin{align}
\label{kawa1}
Ind(\hat{M}, \hat{g}) = \frac{1}{2} ( 15 \chi_{orb}(\hat{M}) + 29 \tau_{orb} (\hat{M})) 
+ N_{\Gamma}',
\end{align}
where $N_{\Gamma}'$ is a correction term depending only upon the 
oriented conjugacy class of the group action. 
The quantity $\chi_{orb}$ is the orbifold Euler characteristic defined by 
\begin{align}
\chi_{orb} = \frac{1}{8 \pi^2} 
\int_{\hat{M}} \left( |W|^2  - \frac{1}{2} |Ric|^2 + \frac{1}{6} R^2 \right) dV_{\hat{g}},
\end{align}
where $W$ is the Weyl tensor, $Ric$ is the Ricci tensor, and 
$R$ is the scalar curvature. The quantity
$\tau_{orb}$ is the orbifold signature defined by 
\begin{align}
\frac{1}{12 \pi^2} \int_{\hat{M}}  \left( |W^+_g|^2 - |W^-_g|^2 \right) dV_{\hat{g}}.
\end{align}
We have the orbifold signature formula
\begin{align}
\label{signature}
\tau(\hat{M}) = \tau_{orb}(\hat{M}) - \eta( S^3 / \Gamma ),
\end{align}
where $ \Gamma \subset {\rm{SO}}(4)$ is the orbifold group at $p$ 
and $\eta(  S^3 / \Gamma)$ is the 
$\eta$-invariant. 
The $\eta$-invariant for the ADE-type singularities is 
written down in \cite{Nakajima}, but we do not require this.
The Gauss-Bonnet formula in this context is
\begin{align}
\label{GB}
\chi(\hat{M}) = \chi_{orb}(\hat{M}) + 1 -  \frac{1}{|\Gamma|}.
\end{align}
See \cite{Hitchin} for a nice discussion of the formulas \eqref{signature}
and \eqref{GB}.

The quantity  $15 \chi_{orb}(\hat{M}) + 29 \tau_{orb} (\hat{M})$ may then be written 
as follows 
\begin{align*}
15 \chi_{orb}(\hat{M}) + 29 \tau_{orb} (\hat{M})  
= 15 \chi(\hat{M}) + 29 \tau(\hat{M})  
- 15 \Big( 1 - \frac{1}{|\Gamma|} \Big)  + 29 \eta(S^3 / \Gamma).
\end{align*}
Consequently, Kawasaki's formula \eqref{kawa1} becomes 
\begin{align}
\label{kawa2}
Ind(\hat{M}, \hat{g}) = \frac{1}{2} ( 15 \chi(\hat{M}) + 29 \tau (\hat{M})) 
+ {N}_{\Gamma},
\end{align}
where ${N}_{\Gamma}$ is a correction term depending only upon the 
oriented conjugacy class of the group action.
The important point is that \eqref{kawa2} holds 
on {\em{any}} anti-self-dual orbifold with one singular point orientation-preserving 
conjugate to type $\Gamma$. Consequently, we can simply 
plug in the examples in Theorem \ref{hypkthm} to determine 
the correction term. 

If the orbifold point is of type $A_1$,
then from Table \ref{hktp}, $\chi(X) = 2$, so
$\chi(\hat{X}) = 3$, and $\tau(\hat{X}) = -1$ (with the 
anti-self-dual orientation), so from Theorem \ref{hypkthm} we have
\begin{align}
4 =  \frac{1}{2} ( 15 \chi(\hat{X}) + 29 \tau (\hat{X})) 
+  {N}_{\Gamma} = 8 + {N}_{\Gamma}, 
\end{align}
so  ${N}_{\Gamma} = -4$. 

If the orbifold point is type $A_n$ for $n \geq 2$, then  
from Table \ref{hktp}, $\chi(\hat{X}) = n +2$, and $\tau(\hat{X}) = -n$,
so from Theorem \ref{hypkthm}, we have 
\begin{align*}
-3n +5 &= \frac{1}{2} ( 15 \chi(\hat{X}) + 29 \tau (\hat{X})) 
+  {N}_{\Gamma} =  \frac{1}{2} ( 15 (n+2) - 29 n) +  {N}_{\Gamma}, 
\end{align*} 
which implies that ${N}_{\Gamma} = 4n - 10$. 

If the orbifold point is type $D_3$, 
then from Table \ref{hktp}, $\chi(\hat{X}) = 5$, and $\tau(\hat{X}) = -3$,
so from Theorem \ref{hypkthm}, we have 
\begin{align*}
-4 &= \frac{1}{2} ( 15 \chi(\hat{X}) + 29 \tau (\hat{X})) 
+  {N}_{\Gamma} 
=  \frac{1}{2} ( 15 \cdot 5 - 29 \cdot 3) +  {N}_{\Gamma}, 
\end{align*} 
which implies that $ {N}_{\Gamma} = 2$. 

Finally, if the orbifold point is of type $D_n$, $n \geq 4$, 
or type $E_n$, $n = 6,7,8$, 
again from Table \ref{hktp} we have $\chi(\hat{X}) = n +2$, and $\tau(\hat{X}) = -n$, 
so from Theorem \ref{hypkthm}, we have 
\begin{align*}
-3n +4 &= \frac{1}{2} ( 15 \chi(\hat{X}) + 29 \tau (\hat{X})) 
+  {N}_{\Gamma} 
=  \frac{1}{2} ( 15 (n+2) - 29 n) + {N}_{\Gamma}, 
\end{align*} 
which implies that ${N}_{\Gamma} = 4n - 11$. 
\end{proof}

\begin{proof}[Proof of Theorem \ref{mainit}]
If $\Gamma \subset {\rm{SO}}(4)$ is a finite subgroup acting freely 
on $S^3$, then we let
$\Gamma$ act on $S^4 \subset \RR^5$ acting as rotations 
around the $x_5$-axis. The quotient $S^4/ \Gamma$ is an
orbifold with two singular points, and the 
spherical metric $g_{S}$ descends to this orbifold. 
The north pole has orbifold group 
$\Gamma$, while the south pole has orbifold group $\tilde{\Gamma}$ 
orientation-reversing conjugate to $\Gamma$. As in 
the proof of Theorem~\ref{mainit}, and as mentioned in 
Remark~\ref{srmk}, Kawasaki's formula yields 
\begin{align}
\label{kawanone}
Ind(S^4 / \Gamma, g_S) = \frac{1}{2} ( 15 \chi(S^4/ \Gamma) + 29 \tau (S^4/ \Gamma)) 
+ N_{\Gamma} +  N_{\tilde{\Gamma}},
\end{align}
It is easy to see that $Ind(S^4 / \Gamma, g_S) = \dim(H^0(S^4 / \Gamma, g_S))$,
$\chi(S^4/ \Gamma) =2 $, and that $ \tau (S^4/ \Gamma) = 0$, so we have 
\begin{align}
\label{kawa3}
\dim(H^0(S^4 / \Gamma, g_S)) = 15 + N_{\Gamma} +  N_{\tilde{\Gamma}}.
\end{align}

We first consider $\Gamma$ of type $A_n$.  We need only consider 
$n \geq 2$ since $n = 1$ is already covered in Theorem \ref{mainit2}
(the $\ZZ / 2\ZZ$-action is orientation-reversing conjugate to itself). 
For $n > 1$, \eqref{kawa3} and Theorem \ref{mainit2} yield 
\begin{align}
5 = 15  + 4n - 10 + N_{\tilde{\Gamma}},
\end{align}
which yields $N_{\tilde{\Gamma}} = - 4n$. 

For $\Gamma$ of type $D_3$, \eqref{kawa3} and Theorem \ref{mainit2} yield 
\begin{align}
5 = 15  + 2 +  N_{\tilde{\Gamma}},
\end{align}
which yields  $N_{\tilde{\Gamma}} = -12$ which is $-4n$ for $n =3$. 

For $\Gamma$ of type $D_n$, $n \geq 4$, or type $E_n$, $n = 6,7,8$, we have 
\begin{align}
4 = 15 + 4n - 11 +  N_{\tilde{\Gamma}}, 
\end{align}
which again yields $N_{\tilde{\Gamma}} = - 4n$.
\end{proof}

\begin{proof}[Proof of Corollary \ref{lbcor}] As mentioned in the Introduction,
as an ALE space, the group 
at infinity of the metric $(\mathcal{O}(-n), g_{\LB})$ is orientation-reversing 
conjugate to type $A_{n-1}$. Consequently, the group at the orbifold 
point of $ \widehat{\mathcal{O}(-n)}$ is orientation-preserving 
conjugate to type $A_{n-1}$. Theorem \ref{mainit} yields the index
\begin{align}
Ind(\widehat{\mathcal{O}(-n)}  , \hat{g}_{\LB})  = \frac{1}{2} \Big( 15 \chi(\widehat{\mathcal{O}(-n)}  ) + 
29 \tau(\widehat{\mathcal{O}(-n)} \Big) - 4 (n-1). 
\end{align}
Using that 
$\chi( \widehat{\mathcal{O}(-n)}) = 3$ and $\tau(\widehat{\mathcal{O}(-n)}) = -1$, 
\eqref{lbind} follows. 
\end{proof}

\begin{proof}[Proof of Theorem \ref{defthm}]
Since $(\mathcal{O}(-n),g_{\LB})$ is scalar-flat K\"ahler, by the 
same argument as given above in the proof of Theorem \ref{hkprop},
and \eqref{h2en}, we have that 
\begin{align}
\label{h20}
\dim(H^2(\widehat{\mathcal{O}(-n)}, \hat{g}_{\LB})) = 0.  
\end{align}
The conformal automorphism group of these metrics is ${\rm{U}}(2)$, which 
has real dimension~$4$, so Corollary \ref{lbcor} implies that 
$\dim(H^1(\widehat{\mathcal{O}(-n)}, \hat{g}_{\LB})) = 4n - 8$.  
From \eqref{h20}, as mentioned above in the 
Introduction, it follows that the moduli space is locally isomorphic to $H^1 / H^0$. 
Finally, since $H^0$ is of dimension $4$, the moduli 
space therefore has dimension at least $4n - 12$. 
\end{proof}
\bibliography{Index_references}

\end{document}